\documentclass[12pt]{amsart}

\usepackage{amsmath}
\usepackage{amssymb}
\usepackage{latexsym}
\usepackage{amsthm,amssymb,amsmath}
\usepackage{geometry}
\usepackage{enumerate}
\def\Om{\Omega}
\def\ov{\overline}
\def\n{\noindent}
\def\cn{\mathbb C^n}
\def\ve{\varepsilon}
\def\nhd{neighbourhood}
\def\si{\sigma}

\def\I{\mathcal I}

 \def\la{\lambda}

\def\al{\alpha}
\def\de{\delta}
\def\F{\mathcal F}

\def\ov{\overline}
\def\De{\Delta}

\def\va{\varphi}
\def\va{\varphi}

\def\und{\underset}

\def\C{\mathbb C}

\newtheorem{theorem}{Theorem}[section]
\newtheorem{cor}[theorem]{Corollary}
\newtheorem{lemma}[theorem]{Lemma}

\newtheorem{prop}[theorem]{Proposition}
\newtheorem{defn}[theorem]{Definition}
\newenvironment{proof*}{\vskip 2mm\noindent {}}{\hfill $\Box$ \vskip 2mm}

\begin{document}

\title{Convergence of multipole Green functions}
\author {Nguyen Quang Dieu and Pascal Thomas}
\address{Nguyen Quang Dieu\\
HaNoi National University of Education\\
136 Xuan Thuy, Cau Giay, Ha Noi, Viet Nam}
\email{ngquang.dieu@hnue.edu.vn}
\address{Pascal J. Thomas\\
Universit\'e de Toulouse\\ UPS, INSA, UT1, UTM \\
Institut de Math\'e\-matiques de Toulouse\\
F-31062 Toulouse, France} \email{pascal.thomas@math.univ-toulouse.fr}
\thanks {This work was essentially done during visits of the first named author at Institute of Mathematics Toulouse in the summers of 2013 and 2014, and of the second named author
at the Viet Nam Institute for Advanced Study in Mathematics in May 2013.
Both author wish to thank the host institutions and the CNRS's Laboratoire 
International Associ\'e ``Formath Vietnam" for the financial support that they received.
This work is supported by the grant 101.02-2013.11 from the NAFOSTED program.
\\
{\it  2010 Mathematics Subject Classification} 32U35, 32U20. }
\maketitle

\begin{abstract}
We continue the study of convergence of multipole pluricomplex Green functions 
for a bounded hyperconvex domain of $\mathbb C^n$, in the case where 
poles collide.  We consider the case where all poles do not converge to the same
point in the domain, and some of them might go to the boundary of the domain.  
We prove that weak convergence will imply convergence in capacity; that
it implies convergence uniformly on compacta away from the poles when
no poles tend to the boundary; and that the study can be reduced, in a sense,
to the case where poles tend to a single point.  Furthermore, we prove that
the limits of Green functions can be obtained as limits of functions of the type
$\max_{1\le i\le 3n} \frac{1}{p} \log |f_i|$, where the $f_i$ are holomorphic functions. 
\end{abstract}


\section{Introduction}
Let $\Om$ be a bounded domain in $\C^n$. We say that $\Om$ is {\it hyperconvex} if its admits a negative continuous plurisubharmonic exhaustion function. A \emph{maximal} plurisubharmonic $g$
function on a domain in $\C^n$ is one that, on a small ball in the domain, lies above any plurisubharmonic 
function that it dominates on the boundary of that ball.  Equivalently,
in the case where $g$ is locally bounded,
$(dd^c g)^n=0$, where $(dd^c)^n$ is  the (complex) Monge-Amp\`ere operator \cite{Lem}, 
\cite{Be-Ta}.

Of course the Monge-Amp\`ere operator, which potentially involves products of
distributions, cannot be defined for an arbitrary locally integrable function.
Bedford and Taylor \cite{Be-Ta} gave a definition for locally bounded plurisubharmonic functions.
Demailly \cite{De1} extended this to plurisubharmonic functions
locally bounded outside of a relatively compact set. We  recall below an important class of plurisubharmonic functions introduced by Cegrell \cite{Ce} on which 
the Monge-Amp\`ere operator behaves nicely.

\begin{defn} \label{Cegrell's class}
Let $\Om$ be a bounded hyperconvex domain in $\cn$. We define $\mathcal E_0 (\Om)$ to be the class of bounded plurisubharmonic functions $u$ on $\Om$ such that $\lim_{z \to \partial \Om} u(z)=0$ and $\int_{\Om} (dd^c u)^n<\infty.$
More generally, $\mathcal F(\Om)$ is the set of plurisubharmonic functions $u$ on $\Om$ such that there exists a sequence $u_j \in \mathcal E_0 (\Om)$ satisfying $u_j \downarrow u$ and $\sup_{j \ge 1} (dd^c u_j)^n<\infty.$
\end{defn}

The Monge-Amp\`ere operator is well defined on $\F(\Om)$ and enjoys basic properties like continuity under monotone sequences, comparison principle, etc.

Green functions on a hyperconvex domain $\Om$ are fundamental solutions of the (complex) Monge-Amp\`ere operator $(dd^c)^n$, i.e. functions $G_a$ such that $(dd^cG_a)^n=\delta_a$, the
Dirac mass at $a\in \Om$,
with zero boundary values; when $\Om$ is
hyperconvex, they are continuous up to the boundary \cite{De1}, \cite{Le}. Since the operator is non-linear when $n \ge 2$, if we want $(dd^cG)^n$ to be a sum of Dirac masses, we cannot
add up Green functions. A function $G$ as above is called multipole Green function, and its
study, initiated by Lelong \cite{Le}, is more delicate.

Let $S:=\{a_1, \dots, a_N\}$ be a finite subset of $\Om.$ The Green function of $\Om$  with the pole set $S$ is defined as follows \cite{Le}:
 $$G_{\Om, S} (z):=\sup \{u(z): u \in PSH^{-} (\Om), u(z)\le \log \vert z-a\vert +O(1),  \forall a \in S\},$$
 where $PSH^{-} (\Om)$ stands for the set of all nonpositive plurisubharmonic functions
 on $\Om$. 
 Multipole Green functions belong to Cegrell's class $\F(\Om)$.

A multipole Green function depends continuously on its poles provided they do not collide. The following result is due to Blocki. In the case of a single pole, this  had been proved by Demailly \cite{De1}.
\begin{prop} \label{blocki}
 If $\Om$ is hyperconvex then the map $(z, p_1, \dots, p_k) \mapsto G_{\Om, (p_1, \dots, p_k)} (z)$  is continuous as  a function defined on the set $\{\ov{\Om} \times \Om^k: z \ne p_j \ne p_k\}.$
 \end{prop}

We are interested in what may happen to limits of sequences of multipole
Green functions.  When poles collide, new singularities will arise, 
generalizing the concept of multiple poles, see \cite{Ma-Ra-Si-Th}.
 
 We establish some terminology about convergence of finite sets. 

\begin{defn}
\label{conv}
 Let $\Om$ be a bounded domain in $\cn$ and $N \ge 1$.
 We say that a sequence $\{S_k\}_{k \ge 1}=\{(a_{1, k}, \dots, a_{N, k})\}_{k \ge 1}  \subset \Om^N$ is  \emph{convergent} if $a_{i, k} \ne a_{j, k}$  for every $1\le i <j \le N$ and if
 $S_k$ converges to an element $S \in  {\ov \Om}^N$;  $\{S_k\}$ is called \emph{interior convergent} if $S \in \Om^N$; it is said to be  \emph{boundary convergent} if  $S \in (\partial \Om)^N$.
 \end{defn}

{\bf Remarks.}
\begin{enumerate}[(a)]
\item
Note that coordinates of the limit point of the sequence $\{S_k\}$  are not necessarily distinct.
\item
We denote by $\pi_N$ the projection $\Om^N \to 2^\Om$ defined by $(a_1, \cdots, a_N) \mapsto \{a_1, \cdots, a_N\} \subset \Om.$
 We will drop the subscript $N$ in case there is no confusion.
 \item
 When the coordinates of the $N$-tuple $S$ are distinct points in $\Om$, 
 by a slight abuse of notation, we will write $G_{\pi(S)}=G_S$.
\item
Renumbering the points as needed, every convergent sequence $S_k$ in $\Om^N$ can be partitioned as
 $S_k=(S'_k , S''_k)$, where $S'_k$ (resp. $S''_k$) is a interior convergent (resp. boundary convergent) in $\Om^{N'}$
(resp. $\Om^{N''}$), with $N'+N''=N$.
\end{enumerate}

Several notions of convergence of functions will occur in this paper. We need a definition, 
which originates in \cite{Xi}. 

\begin{defn}
For a Borel subset $E$ of $\Om$,  the \emph{relative capacity} $C(E, \Om)$ is defined
 \cite{Be-Ta} as
  $$C(E, \Om):=\sup\{\int_E (dd^c u)^n: u \in PSH (\Om), -1 <u<0\}.$$
Given a sequence of functions $\{u_k\}$, we say
$u_k \to 0$ \emph{in capacity} on $\Om$ if for every $\ve>0$ and every Borel set  $F \Subset \Om$ we have
 $C(\{z \in F: \vert u_k (z)\vert>\ve\}, \Om)  \to 0$ as $k \to \infty.$
\end{defn}

It is clear that uniform convergence on compacta of $\Omega$ implies convergence in capacity. 
In fact, uniform convergence on compacta of $\Omega \setminus E$, where $E$ is a compact set 
with zero capacity (e.g. a finite set) is enough to imply convergence in capacity.  On the other 
hand, convergence in capacity implies convergence in measure (in the sense of Lebesgue measure).
If a sequence converges in measure and in the $L^1_{loc}$ topology, both limits must coincide.

By weak compactness in the $L^1_{loc}$ topology, we always have limit
points for a sequence of multipole Green functions $G_{S_k}$.  
Previous work \cite{Ra-Th}
gave sufficient algebraic conditions for such convergence. 

In the same paper, \cite[Theorem 3.1]{Ra-Th} proved that for sequences
of Green functions with all poles tending to one point in $\Om$, convergence in 
the $L^1_{loc}$ topology (the weakest possible in a sense) implies the much stronger
uniform convergence on compacta.
The following generalizes \cite[Theorem 3.1]{Ra-Th}.

\begin{theorem}\label{bootstrap's brother}
Let $\Om \subset \cn$ be a bounded hyperconvex domain
and $\{S_k\}_{k \ge 1}$ be a sequence that converges to $S=(s_1, \dots,s_N) \in \ov{\Om}^N$.
Suppose that $G_k:=G_{\Om, S_k}$ converges in $L^1_{loc}$ to a plurisubharmonic function $g$ on $\Om$.
Then the following assertions hold:
\begin{enumerate}[(a)]
\item
$G_k$ converges in capacity to $g$ on $\Om.$
\item
$g$ is continuous and maximal plurisubharmonic on $\Om \setminus \pi (S) $, and
 $\lim_{z \to  \partial \Om} g(z)=0$. 
\item
$(dd^c g)^n=\sum_{a \in \pi (S) \cap \Om} \nu_a \de_a$, where $\nu_a:=\# \{j\in \{1,\dots,N\}: s_j=a\}.$
\end{enumerate}
Furthermore, if $\{S_k\}_{k \ge 1}$ is an interior convergent sequence,
 the convergence is also uniform on compacta of $\ov\Om \setminus \pi(S)$.
\end{theorem}


The special case of boundary convergent sequences is simpler, in that no assumption
is needed on the convergence of the Green functions. As an immediate consequence
of \cite[Theorem 3.5]{Kh-Hi}, we have:
\begin{prop} \label{capacity convergence}
Let $\Om \subset \cn$ be a bounded hyperconvex domain. 
 Let $\{u_k\} \subset \mathcal F(\Om)$ be a sequence  satisfying the following properties:
\begin{enumerate}[(a)]
\item
 $\sup_{k \ge 1} \int_\Om (dd^c u_k)^n <\infty$;
\item
$\int_E (dd^c u_k)^n \to 0$ for every Borel  set $E \Subset \Om$ as $k \to \infty.$
\end{enumerate}
 Then $u_k \to 0$ in capacity on $\Om$.
  \end{prop}

Applying this to $u_k=G_{S_k}$, where $ \int_\Om (dd^c u_k)^n =N$ for all $k$, we have:
 \begin{cor}
   \label{bdrycv}
  Let $\Om \subset \cn$ be a bounded hyperconvex domain and
   $\{S_k\}_{k \ge 1}  \subset \Om^N$ a boundary convergent sequence. 
  Then $G_{S_k}$ converges to $0$ in capacity.
 \end{cor}
 
 In connection with Corollary \ref{bdrycv}, the following question arises naturally: 
if 
$S_k =(s_{1, k}, \cdots, s_{N, k}) \subset \Om^N$ 
tends to $S =(s_1, \cdots, s_N) \in (\partial \Omega)^N$,
when does $G_{\Om, S_k}$ converge uniformly on compact sets of $\Om$ ? 
Then the conclusions of Theorem \ref{bootstrap's brother}
 and Corollary \ref{criterion} could be strengthened
to convergence on compacta of $\Om\setminus S$ instead of convergence
in capacity. 
By the rough estimate
$$
G_{\Om, S_k} \ge G_{\Om, s_{1,k}}+\cdots+G_{\Om, s_{N, k}},$$
it suffices to consider the case $N=1$.

There is a folklore conjecture that the answer is always positive when $\Om$ is bounded hyperconvex. Nevertheless, it is only proved under the assumptions that $\Om$ admits a negative plurisubharmonic  exhaustion function which is either H\"older continuous \cite{He1} or 
satisfies certain mild conditions about the growth near the boundary \cite{He2}.

A closely related problem is about uniform convergence of $G_{\Om, S_k}$ 
on compacta of $\overline{\Om} \setminus S.$
This question has been studied by Coman (for $N=1$, which is enough).
He shows \cite[Theorem 5]{Co} that if $\Omega$ admits a plurisubharmonic peaking function at the cluster point $s \in \partial \Om$ which is also H\"older continuous near $s$, then $G_{\Om, S_k}$ does converge uniformly on compact sets of $\overline{\Om} \setminus \{s\}$. In particular, this is true at every point $s \in \partial \Om$ if $\Omega$ 
is sufficiently regular (e.g. a real-analytic pseudoconvex domain).

 In the general case, 
 we can apply Theorem \ref{bootstrap's brother}  to give a characterization for convergence in capacity of the sequence $G_{\Om, S_k}.$
\begin{cor} 
\label{criterion}
Let $\Om$ be a bounded hyperconvex domain in $\cn$ and $\{S_k\}_{k \ge 1}$
be a sequence  that converges to $S
=(s_1,\dots,s_N) \in {\ov \Om}^N$. Let
$$G:=(\lim \sup_{k \to \infty} G_{\Om, S_k})^*.$$
Assume that
$$(dd^c G)^n (\pi (S) \cap \Om) \ge \# \{j: s_j \in \Om\}.$$
Then $G_{\Om, S_k}$ converges in capacity to $G$ on $\Om$ as $k \to \infty$.

If $S_k$ is interior convergent, the convergence is locally uniform
on $\ov\Om \setminus \pi (S)$.
\end{cor}

Finally, for interior convergent sequences,
we can reduce the study of convergence of multipole Green functions
to what happens in the case of a single limit point for the poles by using systematically
the next result, which shows how we can break up the limit set $\pi(S)$ into
smaller pieces.

\begin{prop}
\label{split}
Let $\Om \subset \cn$ be a bounded hyperconvex domain and $\{S_k\}$ be an
interior convergent sequence in $\Om^N$. Assume that $S_k=(S'_k , S''_k)$
  where $\{S'_k\}_{k \ge 1}$ and $\{S''_k\}_{k \ge 1}$  are interior convergent sequences of 
 $\Om^{N'}$ and $\Om^{N''}$ respectively ($N'+N''=N$).
  Suppose that $S'_k \to a' \in \Om^{N'}, {S''}_k \to a'' \in \Om^{N''}$,
  and that $\pi_{N'}(a')\cap \pi_{N''}(a'')=\emptyset$. Then the following statements are equivalent:
\begin{enumerate}[(a)]
\item 
$G_{\Om, S_k}$ converges in $L^1_{loc}$  on $\Om \setminus \pi_N (S)$.
\item 
$G_{\Om, S_k}$ converges locally uniformly on $\Om \setminus \pi_N (S)$.
\item 
The two sequences  $G_{\Om, S'_k}$ and $G_{\Om, S''_k}$ converges in $L^1_{loc}$ on $\Om \setminus \pi_{N'} (S')$ and $\Om \setminus \pi_{N''} (S'')$ respectively.
\item 
The two sequences  $G_{\Om, S'_k}$ and $G_{\Om, S''_k}$ converges locally uniformly on $\Om \setminus \pi_{N'} (S')$ and $\Om \setminus \pi_{N''} (S'')$ respectively.
\end{enumerate}
Furthermore, when convergence occurs, if we write 
$g:=\lim_k G_{\Om, S_k}$, $g':=\lim_k G_{\Om, S'_k}$, 
 $g'':=\lim_k G_{\Om, S''_k}$, then we have the following relation:
 $$
 g=\sup\left\{ u\in PSH^-(\Om): u\le g'+O(1), u\le g''+O(1)\right\}.
 $$
\end{prop}

For a finite set $S\subset \Om$, we will denote by $\I_{\Om, S}$ the ideal of
holomorphic functions on $\Om$ which vanish on the set $S$. The
$p$-th power of that ideal, $\I^p_{\Om, S}$, is the the ideal of
holomorphic functions in $\Om$ which vanish to order
at least $p$ on the set $S$.

It is natural to ask how close a pluricomplex Green function 
with pole set $S$ is to being a maximum of 
functions of the form $\frac1p \log|f|$, where $f\in \I^p_{\Om, S}$.  
A version of this in the framework of the convergence question
is given below; it is inspired by  \cite[Theorem 1.1]{Ni}. First we need
to define a slightly more restrictive class of domains.

\begin{defn} \label{strict hyperconvex}
 A bounded domain $\Om$ in $\mathbb C^n$ is said to be \emph{strictly hyperconvex} if
 there exist a bounded open \nhd\ $U$ of  $\ov{\Om}$
 and a real valued continuous plurisubharmonic function $\rho$ on $U$ such that
 $\Om=\{z \in U: \rho(z)<0\}.$
 \end{defn}
This type of domains, under slightly stronger condition, was introduced earlier in \cite{Ni}.
Note that there exist smoothly bounded pseudoconvex domains which do not have a  Stein \nhd\  basis (e.g., worm-domains), in particular, such domains are hyperconvex but not strictly hyperconvex.

\begin{theorem} \label{Nivoche}
 Let $\Om$ be a bounded strictly hyperconvex domain in $\cn$
and $\{S_k\}$ be an interior convergent sequence $\Om^N$ that converges to $S \in \Om^{N}$.
Suppose that $G_{\Om, S_k}$ tends to $g  \in PSH^{-} (\Om)$ in $L^1_{loc} (\Om).$
Then  for every $\ve>0$,
there exist a hyperconvex domain $\Om_\ve$ containing $\ov{\Om}$
such that for every compact $K \subset \Om \setminus S$, we can find $p \ge 1$ and
 a collection of holomorphic functions
$\{(f_{1, k},\cdots, f_{3n, k})\}_{k \ge 1}$ where $f_{j, k} \in \I^p_{\Om_\ve, S_k}$ satisfying the following conditions:

\n
$(i)$ $ \Vert f_{j, k}\Vert_{\ov \Om} <1;$

\n
$(ii)$ For every $k$ large enough, the following estimate holds on $K$
$$g-\ve <\frac1{p} \max \{\log \vert f_{1, k} \vert, \cdots, \log \vert f_{3n, k}\vert \} <g+\ve.$$

\n
$(iii)$ The common zero set in $\Om_\ve$ of $f_{1, k}, \cdots, f_{3n, k}$ coincides exactly with $S$.
\end{theorem}

\n
In the converse direction we have the following partial result,
which says that if a sequence of functions of the correct type
converges, then it must converge to the limit of the corresponding
Green functions.

\begin{prop} \label{converse}
 Let $\Om$ be a bounded hyperconvex domain in $\cn$ and
 $\{S_k\}_{k \ge 1}$ be a sequence in $\Om^N$ that converges to $S \in \Om^N$. Assume that there exist a sequence $\{p_k\}_{k \ge 1}$ of positive integers, a collection
$\{(f_{1, k}, \cdots, f_{n_k, k})\}_{k \ge 1}$ of holomorphic functions on $\Om$ satisfying the following properties.

\n
$ (i) \Vert f_{j, k}\Vert_{\ov \Om} <1$ for every $j, k$.

\n
$ (ii) f_{j, k} \in \I^{p_k}_{S_k}$ for every $k$ and $1 \le j \le n_k.$

\n
$(iii)$ The sequence $u_k :=\frac1{p_k} \max \{\log \vert f_{1, k} \vert, \cdots, \log \vert f_{n_k, k}\vert \}$
converges in $L^1_{loc} (\Om)$ to $u \in \mathcal F (\Om)$ as $k$ tends to $\infty.$

\n
$(iv) \int_\Om (dd^c u)^n  \le N.$

Then the sequence
$G_{\Om, S_k}$ converges to $u$ locally uniformly on $\ov{\Om} \setminus \pi(S).$
\end{prop}

\section{Proofs of Theorem \ref{bootstrap's brother},
Corollary \ref{criterion}, and Proposition \ref{split}}

Our first objective is to show how, for the very special case of Green functions
with interior convergent pole sets, 
$L^1_{loc}$ convergence implies uniform convergence on compacta. 
In order to do so, we will need a property of uniform continuity
that will first require a lemma about variation of domains.

\begin{lemma}\label{variation of domain}
 Let $\Om$ be a bounded hyperconvex domain in $\C^n$ and $E$ be a finite subset of $\Om.$ Then for every $\de>0$, there exists $r_0>0$ and a relatively compact hyperconvex subdomain $\Om'$ of $\Om$ such that for every pole set $E' \subset \cup_{a \in E} \mathbb B(a, r_0)$ we have
$$G_{\Om', E'} (z)  \le G_{\Om, E'} (z) +\de, \ \forall z \in \Om'.$$
\end{lemma}

\begin{proof}
 Choose a compact hyperconvex subdomain $\Om'$ of $\Om$ such that for every $\al \in E$ we have
$$G_{\Om, \al} (z)>- \frac{\de}{2N}, \  \forall z \in \partial \Om',$$
where $N:=\#E$.
By continuity with respect to the pole of the Green functions (cf. Proposition \ref{blocki}) we can find $r_0>0$ so small that
for every $a \in  \cup_{\al \in E} \mathbb B(\al, r_0)$ 
$$
G_{\Om, a} (z) >- \frac{\de}{N}, \  \forall z \in \partial \Om'.
$$
Thus for $E' \subset \cup_{\al \in E} \mathbb B(\al , r_0)$ with $\# E' =N$ we have
$$
G_{\Om, E'} (z) \ge \sum_{\al \in E'} G_{\Om, \al} (z) >-\de,
\ \forall z \in \partial \Om'.
 $$
It follows that $G_{\Om', E'} \le G_{\Om, E'} +\de$ on $\partial \Om'$.
Define $\hat G= \max \{G_{\Om', E'}-\de, G_{\Om, E'} \}$ on $\Om'$ and
$\hat G =G_{\Om, E'}$ on $\Om \setminus \Om'$. It follows from the definition of Green function that
$\hat G \le G_{\Om, E'}$ on $\Om.$ By restricting this inequality to $\Om'$, we obtain the desired estimate.
\end{proof}

\begin{lemma}\label{uniform continuity}
 Let $\Om \subset \cn$ be a bounded hyperconvex domain
and $\{S_k\}_{k \ge 1} \subset \Om^N$ be a sequence 
that  converges to $S \in \Om^N$.
Then  for every fixed point $z_0 \in \Om \setminus \pi(S)$ and $\de>0$,
there exist $r_0>0, k_0 \ge1$ such that for every $z', z'' \in  B(z_0, r_0)$ and $k \ge k_0$ we have
$$\vert G_{\Om, S_k} (z') -G_{\Om, S_k} (z'') \vert<\de.$$
\end{lemma}

\begin{proof} Let $z\cdot \bar w:= \sum_k z_k\bar w_k$ stand for the Hermitian inner product
in $\C^n$. We pick $u \in \mathbb S^{2n-1}:= \{v \in \mathbb C^n : \|v\|=1\}$ such that
for any $a\in\pi(S)$, the orthogonal projection of $a$ to $z_0+\mathbb C u$,
$\pi_{z,u}(a):= z_0 + \left( (a-z_0)\cdot \bar u \right) u$ is different from $z_0$.
This is possible since $\bigcup_{a\in\pi(S)} \{ u \in \mathbb S^{2n-1}: (a-z_0)\cdot \bar u =0\}$
is a subvariety of real codimension $2$ of $\mathbb S^{2n-1}$.

Let $r_1:= \min_{a\in\pi(S)} \left| (a-z_0)\cdot \bar u \right|$.  For $k\ge k_1$,
$\min_{a\in\pi(S_k)} \left| (a-z_0)\cdot \bar u \right| \ge \frac23 r_1$.  Let $r_0:= \frac13 r_1$.
If we take $z'\in B(z_0,r_0)$, then
$$
\left| (z'-a)\cdot \bar u \right| = \left|(z'-z_0)\cdot \bar u  -(a-z_0)\cdot \bar u \right|
\ge r_0.
$$
We define for $k \ge 1$
$$P_k (z):=\prod_{a \in S_k} \left( (z-a)\cdot \bar u \right),
$$
and for $k\ge k_1$, $$ \Phi_k (z):=z+\frac{P_k (z)}{P_k (z')}(z''-z').$$
Note that $\left|\frac{P_k (z)}{P_k (z')}\right| \le C:=\left(2 \mbox{diam}(\Omega)/r_0\right)^{\#(\pi(S))}$, and so $\left| \Phi_k (z)-z\right| \le 2Cr_0$, for any $z\in \Omega$.

Using Lemma \ref{variation of domain} , we can find  a relatively compact hyperconvex subdomain $\Om'$ of $\Om$ such that $z_0 \in \Om'$, and for every $k$ sufficiently large,
\begin{equation} G_{\Om', S_k} (z)  < G_{\Om, S_k} (z) +\de, \ \forall z \in \Om'. \end{equation}
Now, reducing the value of $r_0$, if needed, we may assume that $B(z_0, r_0) \subset \Om'$ and
for every $z', z'' \in \mathbb B(z_0, r_0)$ we have $\Phi_k (\Om') \subset \Om.$
Then, by the decreasing property under holomorphic maps of the Green functions (noting that $\Phi_k$ fixes $S_k$) we get
\begin{equation} G_{\Om, S_k} (z'') \le G_{\Phi_k (\Om'), S_k} (z'') \le G_{\Om', S_k} (z'). \end{equation}
Combining (2) and (3) we arrive  at
$$G_{\Om, S_k} (z'') <G_{\Om, S_k} (z')+\de.$$
By exchanging the roles of $z', z''$ we obtain the desired estimate.
\end{proof}

\begin{lemma} \label{uniform convergence}
 Let $\Om, \{S_k\}_{k \ge 1}$ and $S$ be as in Lemma \ref{uniform continuity}.
Assume that $G_{\Om, S_k}$ tends to a plurisubharmonic function $g$ in $L^1_{loc} (\Om)$. Then the convergence is uniform on every compact subset of $\ov{\Om} \setminus \pi(S).$
\end{lemma}

\begin{proof}
First we claim that  the uniform convergence holds  on compact subsets of $\Om \setminus \pi(S)$.
Assume this is is false, then there exist  a compact subset $K$ of $\Om \setminus S,$ a constant $\de>0,$
two sequences $n_k, m_k \uparrow \infty$ and a sequence $\{z_k\} \subset K$ such that
\begin{equation}\vert G_{\Om, S_{n_k}} (z_k) -G_{\Om, S_{m_k}} (z_k) \vert >\de.\end{equation}
By compactness of $K$, we may assume that $z_k \to z^* \in K.$
By passing to subsequences we can suppose further that both sequences $G_{\Om, S_{n_k}}$
and $G_{\Om, S_{m_k}}$ converge {\it pointwise} to $g$ on a dense subset of $\Om.$
According to Lemma \ref{uniform continuity}, there exists $r_0>0$ such that for every $z', z'' \in \mathbb B(z^*, r_0)$ and $k$ large enough
we have
 \begin{equation}
 \vert G_{\Om, S_{n_k}} (z') -G_{\Om, S_{n_k}} (z'') \vert <\de/3, \vert G_{\Om, S_{m_k}} (z') -G_{\Om, S_{m_k}} (z'') \vert <\de/3 .
 \end{equation}
 Choose a point $w \in B(z^*, r_0)$ such that for $k$ large enough we have
 \begin{equation}\vert G_{\Om, S_{n_k}} (w) -G_{\Om, S_{m_k}} (w) \vert <\de/3 .\end{equation}
By applying (5) with $z'=z_k, z''=w$, together with (4) and (6) we reach a contradiction.
This proves our claim. 

To get the convergence up to the boundary, we need to use a property of
sequences of Green functions that stems from the form of their singularities
and the fact that they are maximal plurisubharmonic outside their poles.
We follow \cite[Lemma 4.5]{Ma-Ra-Si-Th}.

\begin{lemma} 
\label{equicontinuity}
 Let $\Om \subset \cn$ be a bounded hyperconvex domain and $\{S_k\}_{k \ge 1}$ be  sequence  
  that converges to
 $S \in \Om^N$. Denote $G_k:=G_{\Om, S_k}$.
 Assume that there exist  constants $\de_0>0, C>0$ with the following property:  For every $\de \in (0, \de_0]$ we can find  $k(\de)>0$
 such that for every $k_1>k_2>k(\de)$ and $z \in \cup_{a \in \pi (S)} \partial B (a, \de)$ we have
 $$\vert G_{k_1} (z)-G_ {k_2} (z)\vert \le C.$$
 Then the sequence $\{G_k\}$ converges uniformly on compact sets of $\ov \Om \setminus \pi(S)$.
 \end{lemma}
  \begin{proof}
  We can assume that $\ov{B} (a, \de_0) \cap \ov{B} (a', \de_0) =\emptyset$ for $a \ne a' \in \pi(S)$.  Given a compactum $K\subset \ov \Om \setminus \pi(S)$, we have
  uniform bounds on our Green functions given by $0\ge G_k \ge \sum_{s\in S_k} G_{\Om,s}$,
  so that to prove a uniform Cauchy condition, it will be enough to prove that for 
  any $\eta \in (0, 1)$, $z\in K$, there exists $k(\eta)$ such that for 
  $k_1,k_2\ge k(\eta)$, 
  \begin{equation}
  \label{multineq}
  (1+\eta) G_{k_1} (z) \le G_{k_2} (z) \le (1-\eta) G_{k_2} (z).
  \end{equation}

  Near any of its poles $a\in E$, a Green function verifies 
  $G_{\Om,E}(z) \le \log\|z-a\| - \log r$, if $ B(a,r)\subset \Om$. 
  Fix $\eta \in (0, 1)$, using the hypothesis and this upper bound, 
  we can find $\de(\eta)$ so small such that for $\de \in (0, \de(\eta))$ there exists $k( \de)$
  such that if $k_1>k_2>k( \de)$, 
  the  inequalities  \eqref{multineq} hold for
  $z \in \cup_{a \in \pi(S)} \partial B(a, \de)$.
They continue to hold on $\Om \setminus \cup_{a \in \pi(S)} {\ov B} (a, \de)$ since $G_{k_1}$ and $G_{k_2}$ are both maximal continuous plurisubharmonic there, and tend to 
$0$ near $\partial \Om$.
  \end{proof}
To finish the proof of Lemma \ref{uniform convergence}, simply observe
 that uniform convergence on the union of spheres centered on $\pi(S)$
 yields the estimate needed to apply Lemma \ref{equicontinuity}.
\end{proof}

\begin{proof*}{\it Proof of Theorem \ref{bootstrap's brother}.}
Following Remark (d) after Definition \ref{conv}, 
we write $S_k=(S'_k, S''_k)$, where $S'_k$ and $S''_k$ are interior and boundary convergent sequences  in $\Om^{N'}$ and $\Om^{N''}$, respectively.
Then, for any $z\in \Om$ and $k \ge 1$,
\begin{equation}
\label{stdest}
G_{\Om, S''_k}+G_{\Om, S'_k} \le  G_{\Om, S_k} \le  G_{\Om, S'_k}.
\end{equation}
By  Corollary \ref{bdrycv}, $G_{\Om, S''_k}$ goes to $0$ in capacity,
 in particular $G_{\Om, S''_k}$ converges to $0$ in $L^1_{loc}$. 
Since \eqref{stdest} can be rewritten
$$
G_{\Om, S_k} \le  G_{\Om, S'_k} \le G_{\Om, S_k}- G_{\Om, S''_k},
$$
the assumption of the theorem 
 implies that $G_{\Om, S'_k}$ converges to $g$  in $L^1_{loc}$.
Since $S'_k$ is interior convergent, 
by Lemma \ref{uniform convergence}, the convergence is actually uniform on compact subsets of $\ov{\Om} \setminus \pi(S')$, where $S':=\lim_{k \to \infty} S'_k.$
In particular, we have  $\lim_{z \to  \partial \Om} g(z)=0$, 
and also the last statement of the theorem (which is the case where $S_k=S'_k$).

Since uniform convergence of $G_{\Om, S'_k}$ on compacta of $\Om \setminus \pi(S)$ 
implies its convergence in capacity, \eqref{stdest} implies conclusion (a) of the theorem. 

For (c), it suffices to repeat the reasoning at the end of the proof of   \cite[Theorem 1.1]{Ra-Th}. We omit the details.
\end{proof*}

  In the proof of Corollary \ref{criterion}, we will need the following  properties of functions in the class $\F (\Om)$.
  
\begin{lemma}\label{pluripotential lemma}
Let $u, v \in \F (\Om)$
with $u \le v$. Then the following assertions hold.
\begin{enumerate}[(a)]
\item
$\int_S (dd^c v)^n \le \int_S (dd^c u)^n$ for every pluripolar subset $S$ of $\Om.$
\item 
$ \int_{\Om} (v-u)^n (dd^c w)^n \le \int_\Om -w [(dd^c u)^n-(dd^c v)^n],$ for every $w \in PSH (\Om), -1 \le w<0.$
\end{enumerate}
\end{lemma}
\begin{proof}
For (a), see  \cite[Lemma 2.1]{Ah-Ce-Cz-Hi}, and for (b), see  \cite[Proposition 3.4]{Kh-Hi}.
\end{proof}

\begin{proof*}{\it Proof of Corollary \ref{criterion}.}
 We let $g \in PSH^{-} (\Om)$ be an arbitrary limit point of $G_k:=G_{\Om, S_k}$ in $L^1_{loc} (\Om)$.
Then, by Theorem \ref{bootstrap's brother}, $(dd^c g)^n$ is supported on $\pi (S) \cap \Om$ and
$$
\int_\Om (dd^c g)^N  = \sum_{a \in \pi(S) \cap \Om} \nu_a 
=\# \{j\in \{1,\dots,N\}: s_j \in \Om\}.
$$
In particular $g \in \mathcal F(\Om)$. Since $g \le G <0$ on $\Om$ we infer that $G \in \mathcal F(\Om)$.
We claim that $G=g$ on $\Om \setminus \pi (S \cap \Om^N)$. Assume that there exists $z_0 \in \Om \setminus \pi (S \cap \Om^N)$ such that
$G(z_0)>g(z_0)$. Then we can find a $\de>0$ small enough such that:

\n
(i) $B (a, \de) \cap B (a', \de) =\emptyset$ for $a \ne a' \in \pi(S \cap \Om^N)$;

\n
(ii) $B (z_0, \de) \cap X_\de:=\cup_{a \in \pi(S)} \ov{B}(a, \de)$;

\n
(iii) $X_\de$ is polynomially convex.

To see (iii), observe that since $\pi (S)$ is a finite set in $\C^n$, 
there exists a complex line $l$ passing through $0$ such that the orthogonal projections of
the points of $\pi(S)$ on $l$ are all distinct (take $l$ so that it is not
orthogonal to any of the lines defined by pairs of distinct points in $\pi (S)$).
Then for $\delta$ small enough, the projection of $\pi (X_\delta)$ on $H$ consists of a finite number of pairwise disjoint closed discs in $l$.
Thus, by Kallin's lemma \cite[Theorem 1.6.19]{St}, the set $X_\delta$ is polynomially convex. 

This polynomial convexity allows us to choose a plurisubharmonic function $w$ on $\cn$ such that $-1 \le w<0$ on $\Om$, $w$ is strictly plurisubharmonic on 
$B (z_0, \de)$ and $w=-1$ on $X_\de$.
Using Lemma \ref{pluripotential lemma} and the fact that $(dd^c g)^n$ is supported on $\pi (S)$
we obtain
$$\int_\Om (G-g)^n (dd^c w)^n \le \int_\Om w [(dd^c G)^n-(dd^c g)^n] \le 
-(dd^c G)^n (\pi(S))+(dd^c g)^n(\pi(S)) \le 0.$$
This implies that
$$\int_{B (z_0, \de)} (G-g)(dd^c w)^n=0.$$
So $G=g$ a.e. on $B (z_0, \de)$. Since both functions are plurisubharmonic they must coincide on $B (z_0, \de)$.
We get a contradiction. The proof is complete.

 If $S_k$ is interior convergent then, since $G_{\Om, S_k}$ tends
 to $G$ in $L^1_{loc}$, Theorem \ref{bootstrap's brother} tells us
  that the convergence is locally uniform
on $\ov\Om \setminus \pi (S)$.
 \end{proof*}
 
  \begin{proof*} {\it Proof of Proposition \ref{split}.}
  
  By Theorem \ref{bootstrap's brother}, it remains to prove the equivalence of $(b)$ and $(d)$.
  
 $(b) \Rightarrow (d)$
   It suffices to show that $G_{\Om, S'_k}$ converges locally uniformly on $\Om \setminus \pi_{N'} (S')$.
  As in the proof of Theorem \ref{bootstrap's brother},
  we deduce from \eqref{stdest} that for every  $z\in \Om$ and for every $k$,
   $$
   G_{\Om, S_k} (z) \le G_{\Om, S'_k}(z) \le G_{\Om, S_k}(z)-G_{\Om, S''_k}(z).
   $$

Since the $G_{\Om, S''_k}$ are uniformly bounded from below on a small \nhd\  of  $S'_k$
(estimating them from below by sums of one-pole Green functions), by Lemma 
\ref{equicontinuity} we get the desired conclusion.

  $(d) \Rightarrow (b)$ We use the same reasoning as above, together with the following modified form of \eqref{stdest}:
 \begin{equation}
 \label{stdest2}
G_{\Om, S'_k}+G_{\Om, S"_k} \le G_{\Om, S_k} \le \min \{G_{\Om, S'_k}, G_{\Om, S"_k}\}.
 \end{equation}
 To prove the last statement of the theorem, let 
 $$
 \tilde{g}:=\sup\left\{ u\in PSH^-(\Om): u\le g'+O(1), u\le g''+O(1)\right\}.
 $$
Notice that \eqref{stdest2} implies immediately that 
 $g\le \min(g',g'')$, so that $g$ is a candidate for the upper bound which defines
 $\tilde{g}$, therefore $g\le \tilde{g}$. 
 
 To see the reverse inequality, since $g\ge g'+g''$, in a neighborhood of $S'$,
 $g\ge g'-O(1)$, and in a neighborhood of $S''$, $g\ge g''-O(1)$.  So
 $
 \tilde{g} \le \sup\left\{ u\in PSH^-(\Om): u\le g+O(1)\right\}.
 $
 But an easy argument using the maximality of $g$, or the application of \cite[Lemma 4.1]{Ra-Si},
 shows that the latter function is equal to $g$.
 \end{proof*}
 {\bf Remark.} For example, we can apply this result to describe accurately the situations of convergence or non-convergence when $S_k$ consist of $4$ poles and each subset $S'_k, S''_k$ consists of $2$ poles. Indeed, suppose that $S_k=(s_{1,k},\dots,s_{4,k}) \in \Om^4$, that
  $s_{1,k}, s_{2,k} \to a' \in \Om$, and
 $s_{3,k}, s_{4,k} \to a'' \in \Om\setminus\{a'\}$.  Then the Green function
 $G_{S_k,\Om}$ will converge to a limit if and only if the directions
 $[s_{1,k}- s_{2,k}]$ and $[s_{3,k}- s_{4,k}]$ both converge in $\mathbb P^{n-1}\mathbb C$
 \cite[Section 6.1]{Ma-Ra-Si-Th}.

\section{Proof of Theorem \ref{Nivoche} and Proposition \ref{converse}.}

We need a result analogous to Lemma \ref{variation of domain}.

\begin{lemma} \label{continuity Green}
Let $\Om$ be a strictly hyperconvex domain in $\cn ,K$  a  compact subset of  $\Om$ and $N \ge 1$.  Then for every $\de>0,$
 there exists a hyperconvex domain $\Om_\ve$ which contains $\ov{\Om}$ such that
 for every pole set $S \subset K^N$  we have
 $$G_{\Om, S} (z)-\de <G_{\Om_\ve, S} (z) \le G_{\Om, S} (z) \  \forall z \in \Om \setminus S.$$
 \end{lemma}
 
  \begin{proof}
 This lemma in the case where $N=1$ and $K$ is a single point was  essentially proved by Nivoche  \cite[Proposition 2.3]{Ni}. We will use an idea from Theorem 4.3 in \cite{De1}.
For $\ve>0$, we let $\Om_\ve$ be the connected component of $\{z \in U: \rho<\ve\}$ that contains $\Om.$
By choosing $\ve$ small enough we may assure that $\Om_\ve$ is hyperconvex and relatively compact in $U$.
We claim  that if $\ve>0$ is sufficiently small then
$$G_{\Om_\ve, a}>-\de/N, \ \forall  z \in \partial \Om, \forall a \in K.$$
 Choose $r_0>0$ so that $V:=\cup_{a \in K} B (a, r_0)$ is relatively compact in $\Om.$
Let $d$ be the diameter of $U$. Choose a constant $C>0$ big enough such that
$$C \sup_V \rho <\de/N+\log (r_0/d).$$
Choose $\ve>0$ such that
$$C\sup_V \rho-\log (r_0/d) <C\ve<\de/N.$$
Consider the function
\begin{equation}
\hat \rho (z):= \begin{cases} \max \{C(\rho (z)-\ve), \log \frac{\vert z-a\vert}d \}  &  z \in \Om_\ve \setminus B (a, r_0) \\
\log \frac{\vert z-a\vert}d & z \in B (a, r_0).
\end{cases}
\end{equation}
By the choices of $C, \ve$ we can check that $\hat \rho \in PSH^{-} (\Om_\ve)$ and $\hat \rho$ has logarithmic singularity at $a$.
It follows that for every $z \in \partial \Om$ we have
$$G_{\Om_\ve, a} \ge -C\ve >-\de/N.$$
This proves the claim.
So for every $S \subset K^N$ we have
$$G_{\Om_\ve, S} \ge \sum_{a \in S} G_{\Om_\ve, a}>-\de.$$
Fix $S \subset K^N$, we set $\hat G :=G_{\Om_\ve, S}$ on $\Om_\ve \setminus \Om$ while
 $\hat G:=\max\{G_{\Om_\ve, S}, G_{\Om, S}-\de\}$ on $\Om$. Then $\hat G \in PSH^{-} (\Om_\ve)$ has logarithmic singularities at $S$. Therefore $\hat G \le G_{\Om_\ve, S}$. The proof is then easily concluded.
 \end{proof}
 
Theorem \ref{Nivoche} follows essentially from Theorem \ref{bootstrap's brother} and the following lemma which may be of independent interest.

\begin{lemma} \label{Aytuna}
 Let $\Om$ be a bounded hyperconvex domain in $\cn, S \subset \Om^N$,
 and $K$ be a compact subset of $\Om \setminus S$.  
Then  for every $\ve>0$ and every relatively compact sub-domain
$ \Om'$ such that $K \cup S \subset \Om' \subset \Om$
 we can find $p \ge 1, \Om' \subset \Om^\ve \Subset \Om$ and a collection of holomorphic functions
$f_1, \cdots, f_{3n} \in \I^p_{\Om, S}$ having the following properties:

\n
$(i) \Vert f_j\Vert_{\ov {\Om^\ve}} <1$ for every $1 \le j \le 3n;$

\n
$(ii)$ The following estimate holds on $K$
$$G_{\Om, S}-\ve <\frac1{p} \max \{\log \vert f_1 \vert, \cdots, \log \vert f_{3n}\vert \} <G_{\Om, S}+\ve.$$
\end{lemma}

 \begin{proof}
 Choose $\Om^\ve$ such that $ \Om' \Subset \Om^\ve$ and that
 the following estimates hold on $\Om^\ve$
 $$G_{\Om^\ve, S} \le G_{\Om, S}+\ve.$$
 Next, we will prove that there exist an integer $p \ge 1$ and  holomorphic functions
$f_1,\cdots, f_m \in \I^p_{\Om, S}$ having the following properties:

$(i') \Vert f_j\Vert_{\ov {\Om^\ve}} <1,$ for every $1 \le j \le m.$

$(ii')$ The following estimates hold on $K$
$$
G_{\Om, S}-\ve/2 <\frac1{p} \max \{\log \vert f_1 \vert, \cdots, \log \vert f_m\vert \} 
<G_{\Om, S}+\ve/2.
$$

To this end, following an idea of Demailly as in \cite{Ni}, we will use the Ohsawa-Takegoshi extension theorem in the  following special form:
{\it Let $\Om$ be a bounded pseudoconvex domain, $\va \in PSH (\Om)$  and $z \in \Om$. Then for every complex  number $a$, we can find a holomorphic function $f$ in $\Om$ such that $f(z)=a$ and
$$
\int_\Om \vert f(w)\vert^2 e^{-\va (w)} dw \le c_{\Om, n} \vert a\vert^2 e^{-\va (z)},
$$
where $c_{\Om, n}$ depends only on the dimension $n$ and the diameter of $\Om.$
}

Let $r>0$ be the distance between $\partial \Om$ and $\partial \Om^\ve$
and $A>0$ be a constant which is smaller than the volume of the ball with radius $r$ in $\mathbb C^n.$
Choose an integer $p$ so large such that
 $$\ve >-\frac{2}{p} \log (\frac{A}{c_{\Om, n}}).$$
We apply the theorem to $\Om, z_0 \in K, \va =2p G_{\Om, S}$ and
$$a:=\frac{\sqrt{A}e^{pG_{\Om, S} (z_0)}}{\sqrt{c_{\Om, n}}}.$$
 Thus we can find a holomorphic function $f$ on $\Om$ such that
$$\int_{\Om} \vert  f(w)\vert^2 e^{-2pG_{\Om, S} (w)} dw \le A, f(z_0)=a.$$
The first inequality forces $f \in \I^p_{\Om, S}$,
the latter relation and the choice of $p$ implies that
$$\frac1{p} \log \vert f(z_0)\vert >G_{\Om, S} (z_0)-\ve/2.$$
On the other hand, since $G_{\Om, S}<0$ on $\Om$ we also get
$\int_{\Om} \vert f(w)\vert^2 dw<A.$
By the sub-mean inequality applied to the subharmonic functions  $\vert f\vert^2$ over balls of radius $r$ with centers lying on $\partial \Om^\ve,$
we conclude easily from this inequality that $\Vert f\Vert_{\over {\Om^\ve}}<1.$
By a standard compactness argument  we get a finite number of holomorphic functions
$f_1,\cdots, f_m \in \I^p_{\Om, S}$ satisfying $(i')$ and the left inequality in $(ii')$. For the other inequality, it suffices to note that 
by the choice of $\Om^\ve$
$$\frac1{p} \max \{\log \vert f_1 \vert, \cdots, \log \vert f_m\vert \} \le G_{\Om^\ve, S} \le G_{\Om, S}+\ve/2.$$
We are done.

Now it is clear that the proof is complete in the case where $m \le 3n$  by putting together  $(i')$ and $(i'')$ (we can take trivially $f_{m+1}=\cdots=f_{3n}=0$ in this case).
Suppose that $m \ge 3n+1$,
following an idea in the proof  of \cite[Theorem 1]{Ay-Za}, we proceed as follows.
According to \cite[Theorem 1]{Ei-Ev}, there exist polynomials $g_1, \cdots, g_n$ in $\C^n$ such that
 $$
 S=\{z \in \Om: g_1 (z)=\cdots=g_n (z)=0\}.
 $$
Choose polynomials $g_{n+1}, \cdots, g_m$ in $\C^n$ such that any subset of $n$ elements in the collection
$\{g_1, \cdots, g_m\}$ has $S$ as their common zero set.
This can be done by taking $g_{n+1}, \cdots, g_m$ to be sufficiently generic linear combinations of $g_1, \cdots, g_n.$
For $\eta_1, \cdots, \eta_m \in \C,$ we define
$$
h_j:= f_j+\eta_j g_j^p, 1 \le j \le m.
$$
Obviously $h_j \in \I^p_{\Om, S}$,
the key step is to show that we can choose
$\eta_1, \cdots, \eta_m$ so small such that the collection $h_1, \cdots, h_m$ has the
the following additional properties:

\n
$(a) \Vert h_j\Vert_{\ov \Om^{\ve}<1}$ for every $1 \le j \le m$;

\n
$(b) \{z \in \Om \setminus S: \vert h_{j_1} (z) \vert=\cdots=\vert h_{j_{3n+1}} (z)\vert \}=\emptyset$ for every
 $(3n+1)-$ tuple $(j_1, \cdots, j_{3n+1}) \subset  \{1, 2, \cdots, m\};$

\n
$(c) \Vert G_{\Om, S} - w(z)\Vert_K <\ve/4,$ where $w(z):=\frac1{p} \max\{\log \vert h_1 (z)\vert, \cdots, \log \vert h_m (z)\vert\}.$

\n
By the construction of $h_1, \cdots, h_m,$  the properties $(a)$ and $(c)$ are always verified if $\eta_1, \cdots, \eta_m$ are small enough.
For the property $(b)$, since the set of $(3n+1)-$ tuples $(j_1, \cdots, j_{3n+1}) \subset  \{1, 2, \cdots, m\}$ is finite,
for simplicity of notation, it suffices to treat the case where $j_1=1, \cdots, j_{3n+1}=3n+1.$
For $1 \le j \le 3n+1$, denote by $H_j$ the algebraic hypersurface $H_j: =\{z \in \Om: g_j (z)=0\}.$ Let
$$
\Om_j :=(\Om \cap H_1 \cap \cdots \cap H_{j-1})  \setminus (H_j \cup \cdots \cup H_{3n+1}),
$$
$$
\De_j :=\{ (w_j, \cdots, w_{3n+1}) \subset \C^{3n-j+2}: \vert w_j\vert=\cdots=\vert w_{3n+1}\vert\}.
$$
Then by the choice of $g_j$  we have
$$
\Om \setminus S =\bigcup_{1 \le j \le n-1} \Om_j.
$$
It is also easy to check that
$$
\text{dim}_{\mathbb R} (\Om_j) =2(n-j+1), \text{dim}_{\mathbb R} (\De_j)= 3n-j+3.
$$
Consider the map $\Phi_j: \Om_j \times \De_j \to \mathbb C^{3n-j+2}$ defined as
$$
(z, w) \mapsto \Big (\frac{w_j-f_j (z)}{g_j (z)^p}, \cdots, 
\frac{w_{3n+1}-f_{3n+1}(z)}{g_{3n+1} (z)^p}\Big).
$$
Then $\Phi_j$ is a $\mathcal C^\infty$ differentiable map from a real manifold of dimension $5n-3j+3$ into a real manifold of higher dimension $6n-2j+4.$
This implies that  $\Phi_j (\Om_j \times \De_j)$ is a of Lebesgue measure $0$ in $\mathbb C^{3n-j+2}$
for every $j \in \{1, \cdots, n\}.$
Hence for every $\de>0$, there exists $\eta_1, \cdots, \eta_{3n+1}$
such that $\vert \eta_j \vert<\de $ for every $1 \le j \le 3n+1$ and
$$
(\eta_j, \cdots, \eta_{3n+1}) \in \mathbb C^{3n-j+2} \setminus \Phi_j (\Om_j \times \Delta_j), \ \forall 1 \le j \le n-1.
$$
It implies  our claim easily.

Now for every $1 \le r \le 3n$ and $s \ge 1$, we set
$$
h^{(s)}_r (z):=\sum_{1 \le j_1<\cdots<j_r \le m} (h_{j_1} (z)\cdots h_{j_r} (z))^{s\frac{(3n)!}{r}}.
$$
It follows from $(a)$ that $\Vert h^{(s)}_r \Vert_\Om<A$, where $A$ depends only on $n, m$.
Moreover,  each $h^{(s)}_r$  belongs to  $\I^{p'_s}_{\Om, S},$ where $p'_s:=ps(3n)!$
Consider the sequence of functions
$$w_s (z):=\frac1{p'_s} \max\{\log \vert h^{(s)}_1 (z)\vert, \cdots, \log \vert h^{(s)}_{3n} (z)\vert\}.$$
By the same reasoning as in \cite[p. 1735]{Ay-Za}, we will prove that there exists $s_0$ such that
$$\Vert G_{\Om, S} - w_{s_0}\Vert_K<\ve/2.$$
To this end, it suffices to approximate uniformly on $K$ the function $w$ by $w_s$ for $s$ large enough.
We will do the lower bound for $w_s$, the upper bound is easier.
Indeed, fix a point $z_0 \in K$. Then, by the above construction  there exists $r=r(z_0) \le 3n$ and a $r$ tuple $J(z_0):=(j_1, \cdots, j_r)$ such that
$1 \le j_1<\cdots<j_r \le 3n$  and for any $i \not \in J(z_0)$ we have
$$
\la:=\vert h_{j_1} (z_0)\vert=\cdots=\vert h_{j_r} (z_0)\vert >\vert h_i (z_0)\vert.
$$
Choose a small \nhd\ $U_{z_0}$ of $z_0$ such that
$$
d(z_0):=\max_I \und {\xi \in U_z}{\sup} \Big \{\Big \vert \frac{h_{i_1} (\xi) \cdots h_{i_r} (\xi)}{h_{j_1} (\xi) \cdots h_{j_r} (\xi)}\Big \vert\Big \}<1,
$$
where the maximum is taken over all $r-$ tuples $I= (i_1, \cdots, i_r)$ with $I \ne J$.
By continuity we may assume that $U_{z_0}$ satisfies the additional properties
$$
(1-\si) \la <\vert h_j (\xi)\vert, \vert w(z_0)-w(\xi)\vert<\si,  \forall \xi \in U_{z_0}, \forall j \in J(z_0).
$$
Here $\si \in (0,1)$ will be chosen later on.
Then for $\xi \in U_{z_0}$ we obtain the following  estimates
$$\vert h^{(s)}_r (\xi)\vert \ge \vert h_{j_1}(\xi) \cdots h_{j_r} (\xi)\vert^{\frac{s(3n)!}{r}} \Big (1-\sum_{I \ne J}\Big \vert  \frac{h_{i_1} (\xi) \cdots h_{i_r} (\xi)}{h_{j_1} (\xi) \cdots h_{j_r} (\xi)}\Big \vert^s \Big )^{\frac{(3n)!}{r}}$$
$$\ge ((1-\si) \la)^{s(3n)!} (1-2^md(z)^s)^{\frac{(3n)!}r}.$$
Then for $\xi \in U_{z_0}$ and $s >> 1$ we get
$$w_s (\xi) \ge \frac1{ps(3n)!} \log \vert h^{(s)}_r (\xi)\vert \ge \frac1{p}(\log (1-\si)+\log \la)+\frac1{rps}\log (1-2^{m}d(z_0)^s).$$
On the other hand, we also have
$$w(z_0)=\frac1{p} \log \la.$$
Thus, by shrinking $U_{z_0}$ if necessary, we can find $\si \in (0, 1)$ and an integer $s(z_0) \ge 1$ such that
$$w_s (\xi) \ge w(\xi) -\ve/4, \xi \in U_{z_0}, \  \forall s \ge s(z_0).$$
Now a standard compactness argument gives an integer $S(K)$ such that if $s \ge s(K)$ and $\xi \in K$ then
$$w_s (\xi) \ge w(\xi)-\ve/4 \ge G_{\Om, S} (\xi) -\ve/2.$$
The proof is complete.
\end{proof}

\begin{proof*}{\it Proof of Theorem \ref{Nivoche}.}

 By Lemma \ref{continuity Green}, we can find $k_0 \ge 1$ and a hyperconvex domain $\Om_\ve$ containing $\ov{\Om}$ such that the following inequality holds on $\Om \setminus S_k$
for {\it every} $k \ge k_0$
$$G_{\Om, S_k}-\ve/2<G_{\Om_\ve, S_k}.$$
By Theorem \ref{bootstrap's brother}, there exists $k_1$ such that if $k \ge k_1$ then
$$
\Vert G_{\Om, S_k}-g\Vert_K <\ve/2.
$$
Using the above inequalities together with Lemma \ref{Aytuna} (applied to $\Om:=\Om_\ve, S:=S_k$ and $\Om' :=\Om$, we can find $p$ and  holomorphic functions
$f_{1, k}, \cdots, f_{3n, k} \in \I^p_{\Om_\ve, S}$ satisfying the properties $(i), (ii).$
Finally,  by the same perturbation argument i.e., by adding small enough multiples of $g_j^p$, as was done in the proof of Lemma \ref{Aytuna}, we may achieve $(iii)$.

We sketch  the argument. For $1 \le j \le 3n$, denote by $H_j$ the algebraic hypersurface $H_j: =\{z \in \Om_\ve: g_j (z)=0\}.$ Let
$$
\Om_j :=(\Om_\ve \cap H_1 \cap \cdots \cap H_{j-1})  \setminus (H_j \cup \cdots \cup H_{3n+1}).
$$
Consider the map  $\Psi_{j,k}: \Om_j \mapsto \C^{3n-j+1}$ defined as
$$
z \mapsto \Big (-\frac{f_{j,k} (z)}{g_j  (z)^p}, \cdots, -\frac{f_{3n, k} (z)}{g_{3n} (z)^p} \Big).
$$
By counting the dimensions on both spaces we can get 
arbitrarily small $(\ve_{1,k}, \cdots, \ve_{3n,k})$ 
such that $\left( \ve_{j,k}, \cdots, \ve_{3n,k}\right) \notin \Psi_{j,k}(\Om_j)$,
for $1\le j \le 3n$. Then
it suffices to set $\tilde f_{j,k} := f_{j,k} + \ve_j g_j^p$, $1\le j \le 3n$. 
\end{proof*}

\begin{proof*}{\it Proof of Proposition \ref{converse}.}

 Let $g \in PSH^{-} (\Om)$ be an arbitrary limit point of $G_{\Om, S_k}$ in the $L^1_{loc} (\Om)$ topology. By definition of the multipole Green function and $(ii)$ we obtain
$G_{\Om, S_k} \ge u_k$ for every $k.$ It follows that $g \ge u$ a.e. on $\Om$. Since both functions are plurisubharmonic, we have
$g \ge u$ everywhere on $\Om.$ In particular $g \in \mathcal F (\Om)$.
By Theorem \ref{bootstrap's brother}, $(dd^c g)^n$ is supported on $\pi (S)$ and
$$\int_\Om (dd^c g)^n  =N.$$
 Now we apply Lemma \ref{pluripotential lemma} (a) to obtain
 $$(dd^c u)^n (\{a\}) \ge (dd^c g)^n (\{a\}) \ \forall a \in \pi (S).$$
This implies the following chain of inequalities
 $$N \ge \int_{\Om} (dd^c u)^n \ge \int_{\pi (S)} (dd^c u)^n \ge \int_{\pi (S)} (dd^c g)^n = N.$$
 This forces
 $$(dd^c u)^n=(dd^c g)^n=0 \  \text{on}\  \Om \setminus \pi (S),$$
and  $(dd^c u)^n (\{a\}) = (dd^c g)^n (\{a\}) \ \forall a \in \pi (S).$
 In other words, $(dd^c u)^n=(dd^c g)^n$ on $\Om.$
 Putting all this together and use Lemma \ref{pluripotential lemma} (ii) we have in fact $u=g$ on $\Om$.
 This implies that the whole sequence $\{G_{\Om, S_k}\}$ converges to $u$ in $L^1_{loc}$. Applying again Theorem \ref{bootstrap's brother}, we conclude that the convergence occurs locally uniformly away from $\pi(S)$.
\end{proof*}

\end{document}